\documentstyle[12pt]{article}

\newenvironment{proof}{\noindent\bf Proof.\rm}{\hfill $\mbox{\boldmath{$\Box$}}$}
\newcommand{\rr}{\vspace*{2,5mm}}

\newtheorem{defi}{\bf Definition}[section]
\newtheorem{lem}{\bf Lemma}[section]
\newtheorem{theo}{\bf Theorem}[section]

\newtheorem{rem}{\bf Remark}[section]
\newtheorem{exa}{\bf Example}[section]
\newtheorem{exas}{\bf Examples}[section]

\newcommand{\B}{{\cal B}}

\newcommand{\T}{{\cal T}}

\newcommand{\im}{\!\to\!}
\newcommand{\cla}{\ {\scriptstyle \!\supset\!}\ }
\newcommand{\zucc}{\!\succ}
\newcommand{\luk}{\ \!{\scriptstyle > \!\!\to }\ \!}

\title{\large \bf On Generalized $\mbox{\boldmath{$I$}}$--Algebras and 4--valued Modal  Algebras}

\author{Aldo V. Figallo and Paolo Landini}

\date{}

\begin{document}

\maketitle

\begin{center}
Universidad Nacional del Sur, Bah\'{\i}a Blanca, Argentina.
\end{center}
\begin{center}
Universidad Nacional de  San Juan, San Juan , Argentina.
\end{center}

\thispagestyle{empty}

\begin{abstract}
In this paper we establish a new characterization of $4$--valued modal algebras considered by A. Monteiro. In order to obtain this characterization we introduce a new class of algebras named generalized $I$--algebras. This class contains strictly the class of C--algebras defined by Y. Komori as an algebraic counterpart of the infinite--valued implicative {\L}ukasiewicz propositional calculus. On the other hand, the relationship between $I$--algebras and conmutative BCK--algebras, defined by  S. Tanaka in 1975, allows us to say that in a certain sense G--algebras are also a generalization of these latter algebras.
\end{abstract}

\section{\large \bf Introduction}
\vspace*{3mm}

Y. Arai, K. Iseki and S. Tanaka \cite{A.I.T} (see also \cite{YI.KI,KI,KI.ST1,KI.ST2}) defined the class of BCK--algebras as algebras $\langle A,\ast ,0\rangle$ of type $(2,0)$ which satisfies:

\begin{itemize}
\item[(A1)] $((x\ast y)\ast (x\ast z))\ast (z\ast y)=0$,
\item[(A2)] $(x\ast (x\ast y))\ast y=0$,
\item[(A3)] $x\ast x=0$,
\item[(A4)] $0\ast x=0$,
\item[(A5)] $x\ast y=0, \, y\ast x=0$  imply  $x=y$.
\end{itemize}

From A1,$\ldots$,A5 it follows  

\begin{itemize}
\item[(A6)] The relation $x\leq y$ if and only if $x\ast y=0$ is an order on $A$.
\end{itemize}

S. Tanaka \cite{ST}, considered the subclass of conmutative BCK--algebras (or CBCK--algebras) and H. Yutani \cite{HY} proved that these are an equational class of algebras charaterized by the following identities:

\begin{itemize}
\item[(B1)] $(x\ast y)\ast z=(x\ast z)\ast y$,
\item[(B2)] $x\ast (x\ast y)=y\ast (y\ast x)$,
\item[(B3)] $x\ast x=0$,
\item[(B4)] $x\ast 0=0$.
\end{itemize}

S. Tanaka \cite{ST} also proved that every CBCK--algebra is a meet semilattice for the order defined by A6 where the infimum $\wedge$   satisfies:

\begin{itemize}
\item[(B5)] $x\wedge  y=x\ast (x\ast y)$.
\end{itemize}

K. Iseki and S. Tanaka \cite{KI.ST1} proved that every  CBCK--algebra with last element $1$ is a lattice where the supremum $\vee$ verifies:

\begin{itemize}
\item[(B6)] $x\vee y=1\ast((1\ast x)\wedge (1\ast y))$.
\end{itemize}

T. Traczyk \cite{TT} showed that every bounded CBCK--algebra with last element $1$ is a distributive lattice. 

Y. Komori \cite{YK} considered the equational classes of CN--algebras (named Wasjberg algebras by A. J. Rodriguez \cite{AR}) and C--algebras (which we name $I$--algebras). The CN--algebras are the algebraic counterpart of the infinite--valued {\L}ukasiewicz propositional calculus with implication ($\to$) and negation  ($\sim$). The C--algebras are the algebraic version of the implicational part of this calculus. These algebras are defined as follows. 

\rr A Wajsberg algebra (or $W$--algebra) is an algebra $\langle A,\to ,\sim ,1\rangle$ of type $(2,1,0)$ which satisfies:

\begin{itemize}
\item[(W1)] $1\im x=x$,
\item[(W2)] $(x\im y)\im y=(y\im x)\im x$,
\item[(W3)] $(x\im y)\im ((y\im z)\im (x\im z))=1$,
\item[(W4)] $(\sim x\to \, \sim y)\im (y\im x)=1$.(see \cite{F.R.T,YK,AR})
\end{itemize}

An $I$--algebra is an algebra $\langle A,\to ,1\rangle$ of type $(2,0)$ which verifies:

\begin{itemize}
\item[(I1)] $1\im x=x$,
\item[(I2)] $(x\im y)\im y=(y\im x)\im x$,
\item[(I3)] $(x\im y)\im ((y\im z)\im (x\im z))=1$,
\item[(I4)] $((x\im y)\im (y\im x))\im (y\im x)=1$. (see \cite{YK,AF2})
\end{itemize}

An $I^{0}$--algebra is an algebra $\langle A,\to ,1,0\rangle$ of type $(2,0,0)$ which satisfies:

\begin{itemize}
\item[(I5)] $0\im x=1$.
\end{itemize}

We are going to denote by {\bf CBCK}, {\bf W}, {\bf I} and $\mbox{\boldmath{${\rm I}^{0}$}}$ the varieties of algebras described above respectively.

\rr The following results show the relationship between these varieties 
\begin{itemize} 
\item[(CI)] Let $A\in {\bf CBCK}$ be such that it verifies the additional identity:     
\item[] $(x\ast y)\ast ((x\ast y)\ast (y\ast x))=0$.
\end{itemize}

If we define $x\im y=y\ast x$ for all $x,y\in A$, then $\langle A,\to ,0\rangle\in {\bf I}$ and $0$ is the last element of $A$ for the dual order of A6.

\begin{itemize}
\item[(IC)] Let $A\in {\bf I}$. If we define $x\ast y=y\im x$ for all $x,y\in A$  then $\langle A,\ast ,1\rangle \in {\bf CBCK}$ and $1$ is the first element of $A$. 

\item[(WI)] If $\langle A,\to ,\sim ,1\rangle \in {\bf W}$ then the reduct $\langle A,\to ,1\rangle \in{\bf I}$. 

\item[(IW)] If $\langle A,\to ,1,0\rangle \in \mbox{\boldmath{${\rm I}^{0}$}}$ then defining  $\sim x=x\im 0$ for all $x\in A$ we have that $\langle A,\to ,\sim ,1\rangle\in {\bf W}$.

\end{itemize}

In 1978 A. Monteiro introduced the  $4$--valued modal algebras (or ${\rm M}_{4}$--algebras) as algebras $\langle A,\wedge ,\vee ,\sim ,\nabla  ,1\rangle$ of type $(2,2,1,1,0)$ which verify:

\begin{itemize}
\item[(M1)] $x\wedge (x\vee y)=x$,
\item[(M2)] $x\wedge (y\vee z)=(z\wedge x)\vee (z\wedge y)$,
\item[(M3)] $\sim \,\sim x=x$,
\item[(M4)] $\sim (x\vee y)=\,\sim x\,\wedge \sim y$,
\item[(M5)] $\nabla x\, \vee \sim x=1$,
\item[(M6)] $\nabla x \, \wedge \sim x=\,\sim x\wedge x$.\cite{IL1} (see also \cite{AF3,IL2})
\end{itemize}

It is easy to see that every ${\rm M}_{4}$--algebra satisfies:
\begin{itemize}
\item[(M7)] $1\vee x=1$.
\end{itemize}  

From M1, M2, M7, M3, M4 it follows that $\langle A,\wedge ,\vee ,\sim ,1\rangle$ is a De Morgan algebra with last element $1$.
Taking into account [16,17] we have that three--valued {\L}ukasiewicz algebras (or ${\rm L}_{3}$--algebras) are ${\rm M}_{4}$--algebras which satisfy:

\begin{itemize}
\item[(${\rm M}^{'}$6)] $\nabla (x\wedge y)=\nabla x\wedge \nabla y$.
\end{itemize}

I. Loureiro \cite{IL1}, has proved:

\begin{itemize}
\item[(M8)] If $A\in \mbox{\boldmath{${\rm M}_{4}$}}$ is non trivial, then there exists a non empty set $X$ such that $A$ is isomorphic to a subalgebra of $T_{4}^{X}$, where $\T_{4}=\langle T_{4},\vee ,\wedge, \sim ,\nabla ,1\rangle$, $T_{4}=\{0,a,b,1\}$ has the diagram of the figure 1 and $\sim ,\nabla$   are defined by means of the following tables:
\end{itemize}

\vspace*{1mm}
\hspace*{10mm}
\begin{minipage}{8cm}
\setlength{\unitlength}{1cm}
\begin{picture}(4,4)
\put(2,.5){\line(1,1){1.5}}
\put(2,.5){\line(-1,1){1.5}}
\put(2,3.5){\line(-1,-1){1.5}}
\put(2,3.5){\line(1,-1){1.5}}
\put(1.9,.05){$0$}
\put(1.6,-.5){fig. $1$}
\put(.1,1.9){$a$}
\put(3.7,1.9){$b$}
\put(1.9,3.65){$1$}
\put(3.5,2){\circle*{.1}}
\put(2.009,.5){\circle*{.1}}
\put(.5,2){\circle*{.1}}
\put(2.009,3.5){\circle*{.1}}
\end{picture}
\end{minipage}
\begin{minipage}{7cm}
%\vspace*{-16mm}
\begin{tabular}{c|c|c}
$x$ & $\sim x$ & $\nabla x$\\ \hline
$0$ & $1$ & $0$\\
$a$ & $a$ & $1$\\
$b$ & $b$ & $1$\\
$1$ & $0$ & $1$\\
\end{tabular}
\end{minipage}
\vspace*{13mm}

If $A\in \mbox{\boldmath{${\rm M}_{4}$}}$, the operator $\triangle$  is defined by the formula:
\begin{itemize}
\item[(M9)] $\triangle \,x=\,\sim \nabla \sim x$.
\end{itemize}

Now we are going to indicate different operators of implication defined in an ${\rm M}_{4}$--algebra $A$:

\begin{itemize}
\item[(M10)] $x\cla y=\,\sim x\vee y$ (this operation has been defined in the De Morgan algebras \cite{MG.SG}),
\item[(M11)] $x\im y=\nabla \sim x\vee y$, (see \cite{AF3})
\item[(M12)] $x\mapsto y=(x\im y)\wedge (\nabla \,y\vee \sim x)$,
\item[(M13)] $x\zucc y=(x\mapsto y)\wedge ((x\cla y)\im (\triangle \sim x\vee y))$.
\end{itemize} 

Remark that if $A\in \mbox{\boldmath{${\rm M}_{4}$}}$ verifies the Kleene condition $x\wedge \sim x\leq y\,\vee \sim y$, or equivalently if $A\in 
\mbox{\boldmath{${\rm L}_{3}$}}$ then the operators $\mapsto$  and $\succ$  defined by M12 and M13 respectively coincide with the {\L}ukasiewicz implication.

In $\T_{4}$, the operations $\triangle ,\cla ,\to , \mapsto$  and $\succ$  have the following tables:

\vspace*{5mm}
\begin{tabular}{c|c}
$x$ & $\triangle x$\\ \hline
$0$ & $0$\\
$a$ & $0$ \\
$b$ & $0$\\
$1$ & $1$\\
\end{tabular}
\hspace*{5mm}\begin{tabular}{c|cccc}
$\cla$ & $0$ & $a$ & $b$ & $1$\\ \hline
$0$ & $1$ & $1$ & $1$ & $1$\\
$a$ & $a$ & $a$ & $1$ & $1$\\
$b$ & $b$ & $1$ & $b$ & $1$\\
$1$ & $0$ & $a$ & $b$ & $1$\\
\end{tabular}
\hspace*{5mm}\begin{tabular}{c|cccc}
$\to$ & $0$ & $a$ & $b$ & $1$\\ \hline
$0$ & $1$ & $1$ & $1$ & $1$\\
$a$ & $1$ & $1$ & $1$ & $1$\\
$b$ & $1$ & $1$ & $1$ & $1$\\
$1$ & $0$ & $a$ & $b$ & $1$\\
\end{tabular}
\hspace*{5mm}\begin{tabular}{c|cccc}
$\mapsto$ & $0$ & $a$ & $b$ & $1$\\ \hline
$0$ & $1$ & $1$ & $1$ & $1$\\
$a$ & $a$ & $1$ & $1$ & $1$\\
$b$ & $b$ & $1$ & $1$ & $1$\\
$1$ & $0$ & $a$ & $b$ & $1$\\
\end{tabular}

\vspace*{3mm}
\begin{center}
\begin{tabular}{c|cccc}
$\succ$ & $0$ & $a$ & $b$ & $1$\\ \hline
$0$ & $1$ & $1$ & $1$ & $1$\\
$a$ & $a$ & $1$ & $b$ & $1$\\
$b$ & $b$ & $a$ & $1$ & $1$\\
$1$ & $0$ & $a$ & $b$ & $1$\\
\end{tabular}
\end{center}
\vspace*{3mm}

Furthermore in $\T_{4}$ it holds

\begin{itemize}
\item[(L1)] $x\vee y=(x\zucc y)\zucc y$,
\item[(L2)] $\sim x=x\zucc 0$,
\item[(L3)] $\nabla x=\, \sim x\zucc x$,
\item[(L4)] $x\wedge y=\, \sim (\sim x\,\vee \sim y)$.
\end{itemize} 

These identities give the relationships between the variety $\mbox{\boldmath{${\rm W}_{3}$}}$ of $3$--valued Wajsberg algebras and  
$\mbox{\boldmath{${\rm L}_{3}$}}$.

On the other hand by L1,$\ldots$,L4 and M8 it results that ${\rm M}_{4}$--algebras may be characterized by means of the operations $\succ$ and $0$, or $\succ$  and $\sim$.

\rr This fact leads us to pose the following problems:

\rr {\bf Problem 1.}  Axiomatize the matrix $\langle T_{4},\succ ,\sim , D=\{1\}\rangle$.

\rr {\bf Problem 2.} Characterize the $M_{4}$--algebras by means of the operations $\{\succ ,1,0\}$ or 
\hspace*{32mm}$\{\succ ,\sim ,1\}$. 

\rr In this paper we solve the second problem . It is to this end that we introduce a new class of algebras which we name generalized $I$--algebras because it strictly contains the class of $I$--algebras.

\section{\large \bf Generalized $\mbox{\boldmath{$I$}}$--Algebras}
\vspace*{2mm}

\begin{defi}
An algebra $\langle A,\succ ,1\rangle$ of type $(2,0)$ is a generalized $I$--algebra {\rm (}or {\rm G}--algebra{\rm )} if it satisfies:

\begin{itemize}
\item[{\rm (G1)}] $1\zucc x=x$,
\item[{\rm (G2)}] $x\zucc 1=1$,
\item[{\rm (G3)}] $(x\zucc y)\zucc y=(y\zucc x)\zucc x$,
\item[{\rm (G4)}] $x\zucc (y\zucc z)=1$  implies $y\zucc (x\zucc z)=1$.
\end{itemize}
\end{defi}

\begin{exas}
\begin{itemize}
\item[]
\item[{\rm (1)}] The algebra $\langle T_{4},\succ ,1\rangle$ where $T_{4}$ and $\succ$ are defined before is an {\rm G}--algebra but it is not an $I$--algebra.
\item[{\rm (2)}] Let $\langle A,\ast ,0\rangle \in {\bf CBCK}$. If we define $x\zucc y=y\ast x$, for all $x,y\in A$, then $\langle A,\succ ,0\rangle\in {\bf G}$.
\end{itemize}

\end{exas}

\begin{lem} If $A\in {\bf G}$, then it holds:
\begin{itemize}
\item[{\rm (G5)}] $x\zucc x=1$,
\item[{\rm (G6)}] $x\zucc y=1, \, y\zucc x=1$, imply $x=y$,
\item[{\rm (G7)}] $x\zucc (y\zucc x)=1$,
\item[{\rm (G8)}] $x\zucc ((x\zucc y)\zucc y)=1$,
\item[{\rm (G9)}] $x\zucc (z\zucc ((x\zucc y)\zucc y))=1$,
\item[{\rm (G10)}] $x\zucc y=1, \, y\zucc z=1$, imply $x\zucc z=1$,
\item[{\rm (G11)}] $(A,\leq )$ is a partially ordered set, where $\leq$  is given by $x\leq y$ if and only if $x\zucc y=1$,
\item[{\rm (G12)}] $x\leq (x\zucc y)\zucc y$,
\item[{\rm (G13)}] $y\leq (x\zucc y)\zucc y$,
\item[{\rm (G14)}] $x\leq y$ implies $y\zucc z\leq x\zucc z$,
\item[{\rm (G15)}] $x\leq z, \,y\leq z$ imply $(x\zucc y)\zucc y\leq z$,
\item[{\rm (G16)}] $(A,\leq)$ is a join semilattice where the supremun, for all $x,y\in A$ are defined by $x\vee y=(x\zucc y)\zucc y$.
\end{itemize}
\end{lem}

\begin{proof}
\begin{itemize}
\item[(G5)] $x\zucc x=(1\zucc x)\zucc x$, \hfill [G1]
\item[]\ \ \ \ \ \ \ \, $=(x\zucc 1)\zucc 1$, \hfill [G3]
\item[]\ \ \ \ \ \ \ \, $=1$. \hfill [G2]
\item[(G6)] If 

\begin{itemize}
\item[(1)] $x\zucc y=1$,
\item[(2)] $y\zucc x=1$,
\end{itemize}
then  

\begin{itemize}
\item[] $x=1\zucc x$, \hfill [G1]
\item[]\ \, $=(y\zucc x)\zucc x$, \hfill [(2)]
\item[]\ \, $=(x\zucc y)\zucc y$, \hfill [G3]
\item[]\ \, $=1\zucc y$, \hfill [(1)]
\item[]\ \, $=y$.\hfill [G1] 
\end{itemize} 

\item[(G7)] 
\begin{itemize}
\item[(1)] $y\zucc (x\zucc x)=1$,\hfill [G5,G2]
\item[(2)] $x\zucc (y\zucc x)=1$.\hfill [(1),G4]
\end{itemize}

\item[(G8)] It follows from G5 and G8.

\item[(G9)]
 \begin{itemize}
\item[(1)] $x\zucc ((x\zucc y)\zucc y)=1$, \hfill [G8]
\item[(2)] $z\zucc (x\zucc ((x\zucc y)\zucc y))=z\zucc 1=1$, \hfill [(1),G2]
\item[(3)] $x\zucc (z\zucc ((x\zucc y)\zucc y))=1$. \hfill [(2),G3]
\end{itemize}

\item[(G10)] If 

\begin{itemize}
\item[(1)] $x\zucc y=1$,
\item[(2)] $y\zucc z=1$,
\end{itemize}
then 

\begin{itemize}
\item[] $(x\zucc z)=x\zucc (1\zucc z), $\hfill [G1] 
\item[]\ \ \ \ \ \ \ \ \ \, $=x\zucc ((y\zucc z)\zucc z)$, \hfill [(2)]
\item[]\ \ \ \ \ \ \ \ \ \, $=x\zucc ((z\zucc y)\zucc y)$, \hfill [G3] 
\item[]\ \ \ \ \ \ \ \ \ \, $=x\zucc ((z\zucc y)\zucc (1\zucc y))$, \hfill [G1] 
\item[]\ \ \ \ \ \ \ \ \ \, $=x\zucc ((z\zucc y)\zucc ((x\zucc y)\zucc y))$, \hfill [(1)] 
\item[]\ \ \ \ \ \ \ \ \ \, $=1$. \hfill [G9]
\end{itemize}

\item[(G11)] It follows from G5, G6 and G10.
\item[(G12)] It follows from G8 and G11.

\item[(G13)]
\begin{itemize}
\item[(1)] $y\zucc ((x\zucc y)\zucc y)=1$, \hfill [G7]
\item[(2)] $y\leq (x\zucc y)\zucc y$. \hfill [(1),G11]
\end{itemize}

\item[(G14)] If 
\begin{itemize}
\item[(1)] $x\leq y$,
\end{itemize}
then
\begin{itemize}
\item[(2)] $y\leq (y\zucc z)\zucc z$, \hfill [G12]
\item[(3)] $x\leq (y\zucc z)\zucc z$, \hfill [(1),(2),G11]
\item[(4)] $1=x\zucc ((y\zucc z)\zucc z)$, \hfill [(3),G11]
\item[]\ \, $=(y\zucc z)\zucc (x\zucc z)$, \hfill [G4]
\item[(5)] $y\zucc z\leq x\zucc z$. \hfill [(4),G11]
\end{itemize}

\item[(G15)] If
\begin{itemize}
\item[(1)] $x\leq z$,
\item[(2)] $y\leq z$,
\end{itemize}
then 
\begin{itemize}
\item[(3)] $z\zucc y\leq x\zucc y$, \hfill [(1),G14]
\item[(4)] $(x\zucc y)\zucc y\leq (z\zucc y)\zucc y$, \hfill [(3),G14]
\item[(5)] $(x\zucc y)\zucc y\leq (y\zucc z)\zucc z$, \hfill [(4),G3]
\item[(6)] $(x\zucc y)\zucc y\leq z$. \hfill [(5),(2),G11,G1]
\end{itemize}

\item[(G16)]  It follows from G12, G13 and G15.
\end{itemize}
\end{proof}

\begin{defi}
$A\in {\bf G}$ is bounded {\rm (}or ${\rm G}^{0}$--algebra{\rm )} if there exists $0\in A$ such that 
\begin{itemize}
\item[{\rm (G17)}] $0\leq x$ for all $x\in A$.
\end{itemize}
\end{defi}

If $A\in \mbox{\boldmath{${\rm G}^{0}$}}$, we can define the unary operation $\sim$ (called negation) by means of the

\begin{itemize}
\item[{\rm (G18)}] $\sim x=x\zucc 0$.
\end{itemize}

\begin{lem}
If $A\in \mbox{\boldmath{${\rm G}^{0}$}}$ then it holds:

\begin{itemize}
\item[{\rm (G19)}] $\sim \,\sim x=x$,
\item[{\rm (G20)}] $x\leq y$ implies $\sim y\leq \, \sim x$,
\item[{\rm (G21)}] $\sim (\sim x\,\vee \sim y)\leq x$,
\item[{\rm (G22)}] $\sim (\sim x\,\vee \sim y)\leq y$,
\item[{\rm (G23)}] $z\leq x, \, z\leq y$ imply $z\leq \,\sim (\sim x\,\vee \sim y)$,
\item[{\rm (G24)}] $A$ is a meet semilattice where for all $x,y\in A$ the infimum satisfies
\item[] $x\wedge y=\,\sim (\sim x\,\vee \sim y)=(((x\zucc 0)\zucc (y\zucc 0))\zucc (y\zucc 0))\zucc 0$,
\item[{\rm (G25)}] $\sim (x\vee y)=\,\sim x\,\wedge \sim y$,
\item[{\rm (G26)}] $\sim (x\wedge y)=\,\sim x\,\vee \sim y$,
\item[{\rm (G27)}] $(x\zucc y)\vee (y\zucc z)\leq (x\wedge y)\zucc z$.
\end{itemize}

\end{lem}

\begin{proof}
We only prove G27. 

\begin{itemize}
\item[{\rm (G27)}] 
\begin{itemize}
\item[{\rm (1)}] $x\wedge y\leq x$,\hfill [G21,G24] 
\item[{\rm (2)}] $x\wedge y\leq y$,\hfill [G22,G24]
\item[{\rm (3)}] $x\zucc z\leq (x\wedge y)\zucc z$,\hfill [(1),G14]
\item[{\rm (4)}] $y\zucc z\leq (x\wedge y)\zucc z$,\hfill [(2),G14]
\item[{\rm (5)}] $(x\zucc z)\vee (y\zucc z)\leq (x\vee y)\zucc z$.\hfill [(3),(4),G15,G16]
\end{itemize} 

\end{itemize}
\end{proof}

\begin{defi}
$A\in {\bf G^{0}}$ is distributive {\rm (}or ${\rm DG}^{0}$--algebra{\rm )} if it satisfies:
\begin{itemize}
\item[{\rm (DG1)}] $(x\wedge y)\zucc z\leq (x\zucc z)\vee (y\zucc z)$.
\end{itemize}
\end{defi}

\rr If $A\in \mbox{\boldmath{${\rm DG}^{0}$}}$ then it holds 

\begin{itemize}
\item[{\rm (DG2)}] $(x\wedge y)\zucc z\leq (x\zucc z)\vee (y\zucc z)$.
\end{itemize}

\begin{exas}
\begin{itemize}
\item[]
\item[{\rm (1)}] The ${\rm G}^{0}$--algebra $\langle T_{4},\succ ,1\rangle$ is distributive.
\item[{\rm (2)}] Let $\langle A,\succ ,1\rangle \in \mbox{\boldmath{${\rm G}^{0}$}}$ where $A=\{0,a,b,c,1\}$ and $\succ$  is given by the following table
\end{itemize}
\vspace*{5mm}

\begin{center}
\begin{tabular}{c|ccccc}
$\succ$ & $0$ & $a$ & $b$ & $c$ & $1$\\ \hline
$0$ & $1$ & $1$ & $1$ & $1$ & $1$\\
$a$ & $a$ & $1$ & $b$ & $c$ & $1$\\
$b$ & $b$ & $a$ & $1$ & $c$ & $1$ \\
$c$ & $c$ & $a$ & $b$ & $1$ & $1$\\
$1$ & $0$ & $a$ & $b$ & $c$ & $1$\\
\end{tabular}
\end{center}
\vspace*{3mm}

\begin{itemize}
\item[] Then $(A,\leq )$ has the diagram of figure 2 
\end{itemize}

\begin{center}
\setlength{\unitlength}{1cm}
\begin{picture}(4,4)
\put(2,.5){\line(1,1){1.5}}
\put(2,.5){\line(-1,1){1.5}}
\put(2,.5){\line(0,1){3}}
\put(2,3.5){\line(-1,-1){1.5}}
\put(2,3.5){\line(1,-1){1.5}}
\put(1.9,.05){$0$}
\put(1.6,-.5){fig. $2$}
\put(.1,1.9){$a$}
\put(2.18,1.9){$b$}
\put(3.7,1.9){$c$}
\put(1.9,3.65){$1$}
\put(3.5,2){\circle*{.1}}
\put(2.009,.5){\circle*{.1}}
\put(2.009,2){\circle*{.1}}
\put(.5,2){\circle*{.1}}
\put(2.009,3.5){\circle*{.1}}
\end{picture}
\end{center}
\vspace*{5mm}

\begin{itemize}
\item[] A is not distributive, since $(a\wedge b)\zucc c = 1$,  $(a\zucc b)\vee (b\zucc c)=c$  and $1\not=c$.
\end{itemize}

\end{exas}

\begin{theo} If $A\in \mbox{\boldmath{${\rm DG}^{0}$}}$ then $A$ is a De Morgan algebra for the operations $\vee , \wedge ,\sim $ defined in {\rm L1, L4} and {\rm L2}.
\end{theo}

\begin{proof} 
Taking into account 2.5, 2.8, G29 and G25 we only prove the distributive law B7 or equivalently the cancellation law

\begin{itemize}
\item[(CL)] $x\wedge y=x\wedge z$, \,$x\vee y=x\vee z$ imply $y=z$:\\
 If
\begin{itemize}
\item[(1)] $x\wedge y=x\wedge z$,
\item[(2)] $x\vee y=x\vee z$,
\end{itemize} 
then
\begin{itemize}
\item[(3)] $1=(x\wedge z)\zucc z$, \hfill [G24,G11]
\item[]\ \, $=(x\wedge y)\zucc z$, \hfill [(1)]
\item[]\ \, $=(x\zucc z)\vee (y\zucc z)$,\hfill [DG2]
\item[(4)] $1=y\zucc (y\vee z)$, \hfill [G11,G16]
\item[]\ \, $=y\zucc (x\vee z)$, \hfill [(2)]
\item[]\ \, $=y\zucc ((x\zucc z)\zucc z)$, \hfill [G16]
\item[]\ \, $=(x\zucc z)\zucc (y\zucc z)$, \hfill [G4]
\item[(5)] $x\zucc z\leq y\zucc z$, \hfill [(4),G11]
\item[(6)] $1=y\zucc z$, \hfill [(5),(3),G16]
\item[(7)] $y\leq z$. \hfill [(6),G11]
\end{itemize}

Similar we prove
\begin{itemize}
\item[(8)] $z\leq y$.
\end{itemize}

CL It result from (6), (7) and G11.
\end{itemize}
\end{proof}

\section{\large \bf Modal $\mbox{\boldmath{${\rm G}_{4}^{0}$}}$--Algebras}
\vspace*{3mm}

It is easy to prove
\begin{lem} Let $A\in {\bf G}$. The following identities are equivalent:
\begin{itemize}
\item[{\rm (${\rm G}^{'}28$)}] $((x\zucc (x\zucc y))\zucc x)\zucc x=1$,
\item[{\rm (${\rm G}^{'}29$)}] $(x\zucc (x\zucc y))\vee x=1$,
\item[{\rm (${\rm G}^{'}30$)}] $(x\zucc (x\zucc y))\zucc x= x$.
\end{itemize}
\end{lem}

\begin{defi}
$A\in {\bf G}$ or $A\in \mbox{\boldmath{${\rm G}^{0}$}}$ is a ${\rm G}_{4}$--algebra or ${\rm G}_{4}^{0}$--algebra respectively if it satisfies ${\rm G}^{'}28$.
\end{defi}

\begin{exa} Let $C_{n+1}=\{0,\frac1n,\frac2n,\ldots  ,\frac{n-1}{n},1\}$, where $n\geq 3$. For $x,y\in C_{n+1}$ we define \\[1.5mm] $x\zucc y=min \ \{1,1-x+y\}$, then $\langle C_{n+1},\succ ,1\rangle$ is a $G^{0}$--algebra which in not an $G_{4}^{0}$--algebra.
\end{exa}

For each $A\in \mbox{\boldmath{${\rm G}_{4}^{0}$}}$ we define the operators $\luk , \nabla$, by means of the formulas: 
\begin{itemize}
\item[{\rm (${\rm G}^{'}31$)}] $x\luk y=x\zucc (x\zucc y)$,
\item[{\rm (${\rm G}^{'}32$)}] $\nabla x=\,\sim x\zucc x$.
\end{itemize}

\begin{lem} If $A\in \mbox{\boldmath{${\rm G}_{4}^{0}$}}$, then it holds:
\begin{itemize}
\item[{\rm (${\rm G}^{'}33$)}] $\sim x\luk x=\,\sim x\zucc x$,
\item[{\rm (${\rm G}^{'}34$)}] $x\leq \nabla x$.
\end{itemize}
\end{lem}

\begin{proof} 
We only prove ${\rm G}^{'}33$

\begin{itemize}

\item[{\rm (${\rm G}^{'}33$)}] 

\begin{itemize}
\item[(1)] $1=x\vee ((x\zucc (x\zucc 0))$, \hfill [${\rm G}^{'}29$]
\item[]\ \, $=(x\zucc (x\zucc \,\sim x))\zucc (x\zucc \,\sim x)$, \hfill [G18,G16]
\item[]\ \, $=(x\luk \sim x)\zucc (x\zucc \,\sim x)$, \hfill [${\rm G}^{'}31$] 
\item[(2)] $x\luk \sim x\leq x\zucc \,\sim x$. \hfill [(1),G11]
\item[(3)] $x\zucc \,\sim x\leq x\luk  \sim x$.\hfill [G7,${\rm G}^{'}31$,G11]
\end{itemize} 

Now ${\rm G}^{'}33$ follows from (2), (3) and G11.
\end{itemize}
\end{proof}

The following Lemma is an inmediately consequence of the results and definitions given above

\begin{lem} Let $A\in \mbox{\boldmath{${\rm G}_{4}^{0}$}}$. The following identities are equivalent:
\begin{itemize}
\item[{\rm (${\rm G}^{'}35$)}] $\nabla x\luk 0=\nabla x\zucc 0$,
\item[{\rm (${\rm G}^{'}36$)}] $\sim \nabla x\zucc \nabla x=\nabla x$,
\item[{\rm (${\rm G}^{'}37$)}] $\sim \nabla x\vee \nabla x=1$, 
\item[{\rm (${\rm G}^{'}38$)}] $\nabla x\,\vee \sim \nabla x=0$.
\end{itemize} 
\end{lem}

\begin{rem} 
We do not know if ${\rm G}^{'}35$ holds in any ${\rm G}_{4}^{0}$--algebra.
We shall denote by $\mbox{\boldmath{${\rm DG}_{4}^{0}$}}$ the class of distributive ${\rm G}_{4}^{0}$--algebras.
\end{rem}

\begin{defi} $A\in \mbox{\boldmath{${\rm DG}_{4}^{0}$}}$ is a modal ${\rm G}_{4}^{0}$--algebra (or $\mbox{\boldmath{${\rm MDG}_{4}^{0}$}}$--algebra) if it verifies ${\rm G}^{'}35$.
\end{defi}

\begin{theo}
Let $A\in \mbox{\boldmath{${\rm MDG}_{4}^{0}$}}$  and $\vee ,\sim ,\nabla ,\wedge$   be defined by the formulas {\rm L1,$\ldots$,L4}. Then $\B=\langle A,\vee ,\wedge ,\sim ,\nabla ,1)\in {\bf M_{4}}$ and it satisfies $x\im y=(x \mapsto y)\wedge ((x\cla y)\im (\triangle \sim x\vee y))$, where $\triangle ,\cla ,\to ,\mapsto$  are defined by {\rm M9,$\ldots$,M12} respectively.
\end{theo}

\begin{proof}
Taking into account the theorem 2.12 , we only need to prove M5 and M6.
\begin{itemize}
\item[(M5)] $\sim x\vee \nabla x$, \hfill [${\rm G}^{'}32$]
\item[] $=\,\sim x\vee (\sim x\zucc x)$, \hfill [${\rm G}^{'}33$]
\item[] $=\,\sim x\vee (\sim x\luk x)$, \hfill [${\rm G}^{'}29$] 
\item[] $=1$.

\item[(M6)] $\sim x\wedge \nabla x=(\sim x\wedge \nabla x)\vee 0=(\sim x\wedge \nabla  x)\,\wedge \sim 1$, \hfill [M5]
\item[] $=(\sim x\wedge \nabla x)\,\vee \sim (\sim x\vee \nabla x)$, \hfill [G25,G19]
\item[] $=(\sim x\wedge \nabla x)\vee (x\,\wedge \sim \nabla x)$,
\item[] $=(\sim x\wedge \nabla x\wedge x)\vee (\sim x\wedge \nabla x\,\wedge \sim \nabla x)$, \hfill [${\rm G}^{'}34$,${\rm G}^{'}38$]
\item[] $=\,\sim x\wedge x$.
\end{itemize}
\end{proof}

The converse of this results follows inmediately of the representation theorem M8 of I. Loureiro.

\begin{theo} Let $\langle A,\vee ,\wedge ,\sim ,\nabla ,1\rangle \in {\bf M_{4}}$. If we define $\succ$  by means of {\rm M13}, then $\langle A,\succ  ,1\rangle \in \mbox{\boldmath{${\rm MDG}_{4}^{0}$}}$, where $0=\,\sim 1$, and it satisfies {\rm L1,$\ldots$,L4}.
\end{theo}

\subsection*{\large \bf Final Conclusion}

From the above results it follows that an axiomatization for the $4$--valued modal algebras defined by the means of the operations $\{\succ ,\sim ,1\}$ is:
\begin{itemize}
\item[(C1)] $1\zucc x=x$,
\item[(C2)] $x\zucc 1=1$,
\item[(C3)] $(x\zucc y)\zucc y=(y\zucc x)\zucc x$,
\item[(C4)] If $x\zucc (y\zucc z)=1$, then $y\zucc (x\zucc z)=1$,
\item[(C5)] $((x\zucc (x\zucc y))\zucc x)\zucc x=1$,
\item[(C6)] $\sim 1\zucc x=1$,
\item[(C7)] $x\succ \,\sim 1=\,\sim x$,
\item[(C8)] $((x\,\vee \sim y)\zucc z)\zucc ((x\zucc z)\wedge  (y\zucc z))=1$,
\end{itemize}
where $a\vee b$ denotes $(a\zucc b)\zucc b$.

\rr We have named modal algebras to the algebras which satisfy ${\rm G}^{'}35$ because it is in this variety where the operator defined by ${\rm G}^{'}33$ has the modal properties M5 and M6. 

On the other hand, for all $A\in {\bf G}$ we can define the operators $\Rightarrow_{i}$, $i=0,1,\ldots$  by means of the formulas $x\Rightarrow_{0}y=y$, \,$x\Rightarrow_{i+1}y=x\zucc (x\Rightarrow_{i}y)$.
Then we say that $A\in {\bf G}$ is $(n+1)$--valued (or $G_{n+1}$--algebra) if it satisfies the identity:

\begin{itemize}
\item[] $(x\Rightarrow_{n}y)\vee x=1$.
\end{itemize}

We believe that ${\bf G}_{n+1}$--algebras are an interesting generalization of the class of $(n+1)$--valued $I$--algebras. This terminology is analogous to the $(n+1)$--valued Wajsberg algebras, and taking it into account we think that it is more appropiated to call $3$--valued modal algebras to the $4$--valued modal algebras.

\

\begin{minipage}{9cm}
\noindent Aldo V. Figallo\\
Departamento de Matem\'atica\\
Universidad Nacional del Sur\\
8000 Bah\'{\i}a Blanca, Argentina.\\
{\it e-mail: afigallo@criba.edu.ar}
\end{minipage}
\begin{minipage}{7cm}
\vspace{-6mm}
\noindent Paolo Landini\\
Instituto de Ciencias B\'asicas\\
Universidad Nacional de  San Juan\\
5400 San Juan, Argentina.
\end{minipage}

\end{document}